\documentclass[12pt,reqno]{amsart}
\usepackage{graphicx}
\usepackage{float}
\usepackage{psfrag}
\usepackage{pxfonts}
\usepackage{mathrsfs}
\usepackage{color}
\usepackage{amsmath,amssymb}
\usepackage{hyperref}

\numberwithin{equation}{section}

\def \ZZ {{\mathbb Z}}

\theoremstyle{plain}
\newtheorem{maintheorem}{Theorem}

\newtheorem{theorem}{Theorem}[section]

\newtheorem{conjecture}[theorem]{Conjecture}
\newtheorem{cor}[theorem]{Corollary}
\newtheorem{proposition}[theorem]{Proposition}
\newtheorem{lemma}[theorem]{Lemma}
\newtheorem{definition}[theorem]{Definition}

\theoremstyle{remark}

\begin{document}

\author[F. Micena]{Fernando Micena}
\address{Instituto de Matem\'{a}tica e Computa\c{c}\~{a}o,
  IMC-UNIFEI, Itajub\'{a}-MG, Brazil.}
\email{fpmicena82@unifei.edu.br}

\author[R. Ures]{Raúl Ures}
\address{
SUSTech Department of Mathematics and Shenzhen International Center for Mathematics,
1088 Xueyuan Blvd, Nanshan, Shenzhen, Guangdong Province, China, 518055.}
\email{ures@sustech.edu.cn}

\thanks{FM and RU were partially supported by NNSFC 12071202.}
%
%\thanks{R.L. was partially supported by NSF grant DMS-1800241.}
%%%%%%%%%%%%%%%%%%%%%%%%%%
%\baselineskip=18pt              %% DRAFT MODE -- double spaced.
%
%%%%%%%%%%%%%%%%%%%%%%%%%%

%\begin{abstract}
%
%\end{abstract}
%
%\subjclass[2010]{}
%\keywords{}
%
%\maketitle

%\tableofcontents

\title{ On the Ergodicity of Rotation Extensions of Hyperbolic Endomorphisms }
\maketitle

\begin{abstract}
We study the ergodicity of partially hyperbolic endomorphisms, focusing on skew products where the base dynamics are governed by Anosov endomorphisms. For this family, we establish ergodicity and prove that accessibility holds for an open and dense subset. By analyzing the topological implications of accessibility, we demonstrate that conservative accessible partially hyperbolic endomorphisms are topologically transitive. Leveraging accessibility, we further show ergodicity for skew products with
$\mathbb{S}^1$-fibers. Finally, although out the context of rotation extensions, we prove ergodic stability results for partially hyperbolic endomorphisms with $\dim(E^c) = 1.$
\vspace{.5cm}

\end{abstract}

\section{Introduction}

Let
$M$ be a compact connected boundaryless
$C^{\infty}$
  Riemannian manifold. A
$C^1-$endomorphism of
$M$ is defined as a
$C^1-$local diffeomorphism
$f: M \rightarrow M.$

For a compact connected manifold
$M$ and a local homeomorphism
$f: M \rightarrow M$, there exists a constant
$n > 0$ such that
$\#f^{-1}(\{x\}) = n$ for all
$x \in M.$  This constant
$n$, called the degree of
$f$ (denoted
$n = deg(f)$), is uniform across
$M$.

Endomorphism-induced dynamics have attracted significant interest, particularly in three key classes:
expanding maps, Anosov endomorphisms, and partially hyperbolic endomorphisms.

We investigate topological and ergodic properties of partially hyperbolic endomorphisms. While ergodicity for
$C^2-$conservative Anosov diffeomorphisms has been established since the 1990s \cite{An} (via the Hopf argument \cite{Hop}, originally developed by Eberhard Hopf for geodesic flows on compact negatively curved surfaces), the non-invertible setting introduces new challenges.

In 1995, Pugh and Shub [PS1] conjectured at the International Conference on Dynamical Systems (Montevideo) that relaxing hyperbolicity to partial hyperbolicity could establish stable ergodicity as a generic property. Specifically, for
$r \geq 2,$ they formulated:

\begin{conjecture}[Pugh-Shub conjecture] Stable ergodicity is
$C^r-$open and dense among volume-preserving partially hyperbolic diffeomorphisms on connected compact Riemannian manifolds.
  \end{conjecture}

To extend the Hopf argument to partially hyperbolic systems, accessibility property ensuring connectivity via stable/unstable manifolds plays a central role. Formally, a partially hyperbolic diffeomorphism
$f: M \rightarrow M$ is accessible if for any
$x,y \in M,$, there exists a piecewise
$C^1-$path
$\lambda_1\cup\dots\cup \lambda_n$ satisfying:
\begin{itemize}
\item $\lambda_i([0,1])\subset W^\sigma (\lambda_i(0))$ for $\sigma\in \{s,u\}$,
\item $\lambda_1(0)=x$ and $\lambda_n(1)=y$.
\end{itemize}

In 2000, Pugh and Shub \cite{PS} proposed a program to resolve the ergodic stability conjecture, centered on two foundational conjectures:

\begin{conjecture}[Essential Accessibility Implies Ergodicity]
For
$C^2-$volume-preserving partially hyperbolic diffeomorphisms, essential accessibility implies ergodicity.
\end{conjecture}

\begin{conjecture}[Density of Stable Accessibility]
For
$r \geq 2,$ stable accessibility is
$C^r-$dense in the space of partially hyperbolic diffeomorphisms.
\end{conjecture}

In 2007, F. Rodriguez Hertz, J. Rodriguez Hertz, and R. Ures \cite{HHU} established foundational results for systems with
$\dim(E^c) = 1$:

\begin{itemize}
\item For
 $r \geq 2$ , stable accessibility is
$C^r-$dense among volume-preserving partially hyperbolic diffeomorphisms.

\item For
$C^2-$
  volume-preserving partially hyperbolic diffeomorphisms, accessibility implies ergodicity.

\end{itemize}

The noninvertibility of an endomorphism can complicate the study of ergodicity. We typically seek to convert a noninvertible problem into one involving a homeomorphism. For an endomorphism  $f: M \rightarrow M$ we define the orbit space
 $M^f$ as the set of sequences  $\bar{x} = (x_n)_{n \in \mathbb{Z}}$ satisfying $f(x_n) = x_{n+1}.$ We equip this space with the metric $\bar{d}$
$$
\bar{d}(\tilde{x},\tilde{y})=\displaystyle\sum_{i\in \mathbb{Z}}\frac{d(x_i,y_i)}{2^{|i|}},
$$
where $d$ is the Riemannian metric of $M$. The space $(M^f, \bar{d})$ is  compact. We define the shift  as $\bar{f}(\tilde{x}) = \tilde{y} = (y_n)_{n \in \mathbb{Z}},$ whith $y_{n} = x_{n+1}.$ Since $f$ is continuous,  $\bar{f}$ is continuous. We define the projection $p: M^f \rightarrow M$ by $p((x_n)_{n \in \mathbb{Z}}) = x_0,$ which is continuous.

A local diffeomorphism $f: M \rightarrow M$
 is partially hyperbolic if, along every orbit
$\bar{x} = (x_n)_{n \in \mathbb{Z}},$ the tangent bundle splits into uniformly contracting ($E^s$), expanding ($E^u$), and intermediately behaving ($E^c$) subbundles. If
$E^c$
  is trivial,
$f$ is an Anosov endomorphism.
More precisely:

\begin{definition}[(absolute) Partially Hyperbolic Endomorphism]

 A $C^1-$local diffeomorphism $f: M \rightarrow M$ is  \emph{partially hyperbolic} if there exist:
 \begin{itemize}
\item A Riemannian metric
$\langle\cdot,\cdot\rangle$,
\item Constants \mbox{$0<\nu<\gamma_1\leq\gamma_2<\mu$} with $\nu<1,$ $\mu>1$ and $C>1,$
\end{itemize}

such
that for each orbit $(x_n)_{n\in \mathbb{Z}}$ of $f,$  the tangent bundle admits a splitting:
$$T_{x_n}M=E_{x_n}^s\oplus E_{x_n}^c\oplus E_{x_n}^u$$ satisfying for all $i\in \ZZ$, $n\geq 0$, and vectors $v^\sigma\in E^\sigma_{x_i}\,\,(\sigma\in\{s,c,u\})$:

\begin{enumerate}

\item Invariance: $Df_{x_i}(E^{\ast}_{x_i})=E^{\ast}_{x_{i+1}},$  for $\sigma \in \{s,c,u\},$

 \item Uniform contraction: $||Df^n_{x_i}(v^s)||\leq C\nu^n||v^s||,$
 \item Intermediate behavior: $C^{-1}\gamma^n_1||v^c||\leq||Df^n_{x_i}(v^c)||\leq C\gamma^n_2||v^c||,$
 \item Uniform expansion: $C^{-1}\mu^n||v^u||\leq||Df^n_{x_i}(v^u)||.$

\end{enumerate}

\end{definition}

It is clear that  $E^s_{x_i}$ depends solely on $x_0,$  which means there is a unique $E^s$ for each point $x \in M.$  In contrast,  this independence does not  hold for the center bundle  $E^c$ and the unstable bundle $E^u.$

\begin{definition}[Special Partially Hyperbolic Endomorphism] A partially hyperbolic endomorphism $f: M \rightarrow M$  is:
\begin{itemize}
\item {\bf $u-$special} f for every $x \in M,$ there is a unique unstable subbundle $E^u_x,$ independent of the orbit $\bar{x} \in M^f,$ with
$x_0 = x.$
\item {\bf  $c-$special } if an analogous uniqueness holds for the center subbundle $E^c_x$.
\end{itemize}

We say that $f$ is {\bf special} if it is both $c$ and $u-$special.

\end{definition}

The complex structure of the unstable bundles of Anosov endomorphisms was observed independently by Przytycki in \cite{PRZ} and Ma\~n\'e-Pugh in \cite{MP75}.

In fact, we have:

\begin{theorem}[\cite{micsa}]\label{Teo C}
Let  $f:M\rightarrow M$ be a transitive partially hyperbolic endomorphism. For each $\sigma \in \{u,c\}$ the following dichotomy holds:
\begin{itemize}
\item Either $f$ is a $\sigma-$special partially hyperbolic endomorphism.
\item Or there exists a residual subset $\mathcal{R}\subset M$ such that for every $x \in \mathcal{R},$ $T_xM$ admits
infinitely many distinct $\sigma$-directions.
\end{itemize}

\end{theorem}

The case  $\sigma = u$ of Theorem \ref{Teo C} was first established for  Anosov endomorphisms by Micena and Tahzibi \cite{MT}.

While the unstable "bundle" $E^u$ may exhibit complex structure, lifting to the universal cover
$\widetilde M$ simplifies the geometry at the cost of losing compactness.

\begin{proposition}[\cite{micsa}]\label{pro 6}
Let $f:M\rightarrow M$ be a $C^1$ endomorphism. Then, $f$ is a partially hyperbolic endomorphism if and only if every lift $\widetilde{f}:\widetilde M\to \widetilde M$  to the universal cover is a partially hyperbolic diffeomorphism.
\end{proposition}

\begin{cor}
The set of all partially hyperbolic endomorphisms is $C^1-$open in the space of $C^1-$endomorphisms on $M$.
\end{cor}

Our first result is the following:

\begin{maintheorem}[Integrability of the Stable-Unstalbe Splitting]\label{mainA}
Let $f: M \rightarrow M$ be an Anosov endomorphism and $\varphi: M \rightarrow \mathbb{S}^1$ a $C^1-$map. Define the rotation extension $$F_{\varphi}:M \times \mathbb{S}^1 \rightarrow  M \times \mathbb{S}^1, F_{\varphi}(x, \theta) = (f(x), \theta + \varphi(x)),$$  and assume:

\begin{enumerate}

\item The stable foliation $\mathcal{W}^s_f$ is minimal.
\item The extended stable foliation $\mathcal{W}^s_{F_{\varphi}}$ is not minimal.

\end{enumerate}
Then, the splitting $E^s\oplus E^u $ is integrable.
\end{maintheorem}

A straightforward consequence of Theorem \ref{mainA} is the following.

\begin{cor}
Let $F_{\varphi}$ be as in Theorem A. If $F_{\varphi}$ is  $C^1-$smooth and $m-$preserving, then non-integrability of  $E^s\oplus E^u $ is implies  $F_{\varphi}$ is topologically mixing.
\end{cor}

\begin{cor}[Toral Case]\label{corA}
Let $f: \mathbb{T}^2 \rightarrow \mathbb{T}^2$ be an Anosov endomorphism and $\varphi: \mathbb{T}^2 \rightarrow \mathbb{S}^1$ a $C^1-$map. For the rotation extension $$F_{\varphi}:\mathbb{T}^2 \times \mathbb{S}^1 \rightarrow  \mathbb{T}^2 \times \mathbb{S}^1, F_{\varphi}(x, \theta) = (f(x), \theta + \varphi(x)),$$ if  $\mathcal{W}^s_{F_{\varphi}}$ is not minimal, then $E^s\oplus E^u $ is integrable.
\end{cor}

\begin{proof}
Follows from Theorem \ref{mainA} and the minimality of $\mathcal{W}^s_{f}$ on $\mathbb T^2$ \cite[Theorem B]{HH21}.
\end{proof}

%We are interested in investigating ergodicity for partially hyperbolic in the volume-preserving setting.
 Let $m$ denote a $C^{\infty}-$normalized volume on $M.$
\begin{maintheorem}[Ergodicity from Non-integrability]\label{mainB} Let $F_{\varphi}$ be as in Theorem A. If $F_{\varphi}$ is  $C^2-$smooth,  $m-$preserving map, and  $E^s\oplus E^u $ is not integrable, then $(F_{\varphi}, m)$ is ergodic.
\end{maintheorem}

While \cite{AKS} proved an analogous result for an Anosov diffeomorphism acting on the base $\mathbb{T}^2$ using Fourier analysis, our approach adapts  Pesin's theory  to endomorphisms. This framework extends ergodicity results to general base manifolds $M,$
bypassing the need for Fourier techniques.

\begin{maintheorem}[Topological Transitivity via Accessibility]\label{accBrin}  Let $f: M \rightarrow M$ be an accessible,  $m-$preserving partially hyperbolic endomorphism. Then $f$ is topologically transitive.
\end{maintheorem}

This aligns with Brin's classical result for diffeomorphisms \cite{Br}, extended here to endomorphisms.

When $dim E^c = 1,$  accessibility  implies stable accessibility, \cite{He}. Combining this with Theorem \ref{accBrin}, we obtain:

\begin{theorem}[Robust Transitivity] If  $f: M \rightarrow M$ is an $m-$preserving, accessible partially hyperbolic endomorphism with $dim E^c = 1$, then $f$ is $C^1-$robustly transitive among  $m-$preserving $C^1-$endomorphisms.
\end{theorem}

\begin{cor}[Robust Transitivity from Non-Integrability] Let $F_{\varphi}$ be as in Theorem \ref{mainA}. If $E^s \oplus E^u$ is not integrable, then $F_{\varphi}$ is $C^1-$robustly transitive among  all $m-$preserving $C^1-$endomorphisms.
\end{cor}

\begin{maintheorem}[Ergodicity for Expanding Base Maps] \label{mainD} Let $f: M \rightarrow M$ be a $C^2-$expanding map and $F: M \times \mathbb{S}^1 \rightarrow M \times \mathbb{S}^1$ a skew product of the form $F(x,\theta) = (f(x), \theta + \varphi(x)),$ where $\varphi: M \rightarrow \mathbb{S}^1$ is  $C^2-$smooth. If $F$ is:
\begin{enumerate}
\item Topologically transitive,
\item Volume preserving (with $m$ a normalized volume measure),

\end{enumerate}
then $F$ is ergodic.
\end{maintheorem}

%{\color{red} HASTA AC\'A}

Combining Theorems \ref{accBrin} and \ref{mainD} we obtain the following corollary:

\begin{cor}[Ergodicity of Accessible Skew Products]  Let $f: M \rightarrow M$ a $C^2-$expanding map, and let $F: M \times \mathbb{S}^1 \rightarrow M \times \mathbb{S}^1$ a skew product of the form $$F(x,\theta) = (f(x), \theta + \varphi(x)),$$ where $\varphi: M \rightarrow \mathbb{S}^1$ is a $C^2-$map. If $F$ is:

\begin{enumerate}
\item Volume-preserving,
\item Accessible (as a partially hyperbolic endomorphism),
\end{enumerate}
then $F$ is ergodic.
\end{cor}

We now address stable ergodicity--persistence of ergodicity under small perturbations--for partially hyperbolic endomorphisms. Formally:

\begin{definition} A $C^2-$partially hyperbolic endomorphism $f: M \rightarrow M,$ preserving a normalized volume form $m$ is stably ergodic if there is a $C^1-$open set $\mathcal{U}$ among all $C^2$ and $m-$preserving  partially hyperbolic endomorphism $g: M \rightarrow M,$ such that, every $g \in \mathcal{U}$ is ergodic with respect to $m.$
\end{definition}

%The next result is addressed to investigate sufficient conditions to stable ergodicity in the endomorphism setting.

\begin{maintheorem}[Ergodicity Criteria for Partially Hyperbolic Endomorphisms]\label{mainC} Let $f: M \rightarrow M$ be a $C^2-$volume preserving partially hyperbolic endomorphism with:
\begin{itemize}
\item $\dim(E^c)=1,$
\item $\dim(E^s) = d_s \geq 1, \dim(E^u) = d_u \geq 1,$
\item $\displaystyle\int_{M} \lambda^c_f(x) dm(x) > 0$ (positive mean central Lyapunov exponent).
\end{itemize}
Then:

\begin{enumerate}
\item If the stable foliation $\mathcal{F}^s_f$ is minimal, then $f$ is ergodic.
\item If $\mathcal{F}^s_f$ is robustly minimal, then $f$ is stably ergodic.
\end{enumerate}
\end{maintheorem}

For diffeomorphisms, Burns, Dolgopyat, and Pesin \cite{BDP} proved analogous results using accessibility and negative central Lyapunov exponents. Their framework does not require
$\dim(E^c) = 1$.

\begin{theorem}[Theorem 3 of \cite{BDP}]
Let $fM\to M$ be a $C^2$ partially hyperbolic diffeomorphism of a
compact smooth Riemannian manifold $M$, preserving a smooth measure $m.$
If $f$ is accessible and has negative central exponents on a set of
positive measure, then $f$ is stably ergodic.
\end{theorem}

Establishing ergodicity and stable ergodicity for non-invertible systems (endomorphisms) presents unique challenges, particularly in bridging accessibility to ergodic behavior problem of significant recent interest. In his Ph.D. thesis, A. Tyler  \cite{AuT} advanced this program by proving:

\begin{theorem}[Theorem 1.7 of \cite{AuT}] Let $fM\to M$  be  a $C^2$ volume-preserving, center bunched, dynamically coherent, partially hyperbolic endomorphism on a compact smooth Riemannian manifold $M$. Assume:

\begin{enumerate}
\item $f$ has constant Jacobian,
 \item $f$ is essentially accessible.
\end{enumerate}
Then $f$ is ergodic.
\end{theorem}

\section{Basic Preliminaries of Pesin Theory}

Our analysis leverages Pesin Theory for endomorphisms \cite{QXZ}, focusing on the ergodic decomposition of SRB (Sinai-Ruelle-Bowen) measures. The core strategy is to prove that the ergodic decomposition of the
 $f$-invariant measure
$m$ is trivial (i.e.,
$m$ itself is ergodic).

%For us it will be very useful the Pesin Theory in the context of endomorphism, in which the main reference is \cite{QXZ}. Our strategy in the proof of the ergodicity consists in proving that the ergodic decomposition of $m$ is single. In our case, the invariant measure $m$ is a $SRB-$measure, so it will be interesting to understand general facts about $SRB$ measures and their ergodic decompositions.
%
\begin{theorem}[Invariant Measures on Inverse Limits \cite{QXZ}] Let $(M,d)$ be a compact metric space and $f: M \rightarrow M$ a continuous map. For any  $f-$invariant Borel probability measure $\mu$, the exist a unique $\tilde{f}-$invariant Borel probability measure $\tilde{\mu}$ on the inverse limit space $M^f,$ such that: $$\mu(B) = \hat{\mu}(p^{-1}(B)) \text{ for all Borel sets }B\subset M,$$
where $p:M^f\to M$ is the canonical projection
\end{theorem}
%.
%
%
%
%
%
%There is a long and well-developed theory about SRB measures for endomorphisms.
%

The theory of SRB measures for endomorphisms is well-developed \cite{QXZ}. These measures are characterized by absolute continuity of conditional measures on unstable manifolds, generalizing the classical Sinai-Ruelle-Bowen framework to non-invertible systems.

\begin{definition}[Subordinate Partition to $W^u$-manifolds]\label{subpart}
A measurable partition $\eta$ of $M^f$ is subordinate to
$W^u-$manifolds of a system $(f, \mu)$ if for $\hat{\mu}$-a.e. $\tilde{x} \in M^f,$ the atom $\eta(\tilde{x}),$ satisfies:\begin{enumerate}
\item Bijectivity: The projection $p| \eta(\tilde{x}): \eta(\tilde{x})  \rightarrow p(\eta(\tilde{x}))$ is bijective.
\item Embedded Submanifold: There is a $k(\tilde{x})-$dimensional $C^1-$embedded submanifold $W(\tilde{x})\subset W^u(\tilde{x})$ such that:
\begin{itemize}
\item $p(\eta(\tilde{x})) \subset W(\tilde{x}),$
and $p(\eta(\tilde{x}))$ contains an open neighborhood of $x_0 $ in $W(\tilde{x})$ (under the submanifold topology).
\end{itemize}
\end{enumerate}
\end{definition}

Such partitions can be taken increasing, that means $\eta$ refines $\tilde{f}(\eta) .$ Particularly $ p(\eta(\tilde{f}(\tilde{x}))) \subset p(\tilde{f}(\eta(\tilde{x}))) .$

\begin{definition}[SRB Property] Let $f: M \rightarrow M$ be a $C^2-$endomorphism preserving an invariant Borel probability $\nu.$
The measure $\nu$ has SRB property if for every measurable partition $\eta $ of $M^f$ subordinate
to $W^u-$manifolds (as in Definition \ref{subpart}), the following holds: $$p(\hat{\nu}_{\eta{(\tilde{x})}}) \approx m^u_{p(\eta(\tilde{x}))},\text{ for }\hat{\nu}-\text{ a.e. }\tilde{x},$$
where:
\begin{itemize}
\item $\{\hat{\nu}_{\eta{(\tilde{x})}} \}_{\tilde{x} \in M^f}$
is a canonical system of conditional measures of $\hat{\nu}$,
\item $m^u_{p(\eta(\tilde{x}))} $ is the Lebesgue measure on $W(\tilde{x})$, induced by the inherited Riemannian metric.
\end{itemize}
\end{definition}

The Radon-Nikodym derivative relating $\hat{\nu}_{\eta({\tilde{x}})}$ and  $m^u_{p(\eta({\tilde{x}}))}$ is given by:
% $$p(\hat{\nu}_{\eta({\tilde{x}})}) \approx m^u_{p(\eta({\tilde{x}}))}, $$ and
$$\rho^u_f(y) =
\frac{\Delta^u_f(\tilde{x}, \tilde{y})}{L(\tilde{x})},\,\,\,\, y \in p(\eta({\tilde{x}})),$$
where:
$$ \Delta^u_f(\tilde{x},\tilde{y}) = \displaystyle\prod_{k=1}^{\infty} \frac{J^uf(x_{-k})}{J^uf(y_{-k})}, \,\,\,\tilde{x} = (x_k)_{k \in \mathbb{Z}}, \tilde{y} = (y_k)_{k \in \mathbb{Z}} $$
and
$$L(\tilde{x}) = \int_{\eta(\tilde{x})} \Delta^u_f(\tilde{x}, \tilde{y}) d \hat{m}^u_{\eta({\tilde{x}})}(\tilde{y}).$$

\begin{theorem}[\cite{QZ}] \label{pesin2} Let $f:M\to M$ be a $C^2$-endomorphism  with an invariant Borel probability
measure $\mu$ satisfying $\log(|Jf(x)|) \in L^1(M,\mu).$ Then, the entropy
formula
\begin{equation}\label{PesinU}
h_{\mu}(f) = \displaystyle\int_M \displaystyle\sum \lambda^i(x)^{+}m_i(x) d\mu
\end{equation}
holds if and only if $\mu$ is an  SRB measure.
\end{theorem}

By Pesin's Formula, the measure $m$ is SRB for every $C^2-$endomorphism $f$ preserving $m.$

\subsection{Ergodic Decomposition}  Our strategy to establish ergodicity relies on analyzing the ergodic decomposition of the measure. Specifically, for a continuous dynamical system
$f:M\to M$ and an invariant Borel probability measure
$\mu$, the ergodic decomposition theorem guarantees that
$\mu$ can be disintegrated into ergodic components. This decomposition is essentially unique. If the decomposition consists of a single component (i.e.,
$\mu$ itself), then
$\mu$ is ergodic. Formally, this means that any
$f-$invariant set has either full measure or zero measure under
$\mu$, confirming ergodicity.

\begin{theorem}[Ergodic decomposition] Let $M$ be a complete separable
metric space, $f : M \rightarrow M$ be a measurable transformation, and $\mu$  an
$f-$invariant Borel probability measure. There exist
\begin{itemize}
\item A measurable set $M_0  \subset  M$
with $\mu(M_0) = 1$,
\item  A partition $\mathcal{P}$ of $M_0$ into measurable subsets,
\item A family
$\{\mu_{P} : P \in \mathcal{P}\}$ of probability measures on $M,$
\end{itemize}
satisfying:
 \begin{enumerate}
\item  $\mu_P(P) = 1$ for $\overline{\mu}$-almost every $P \in \mathcal{P},$
\item For every measurable set $E \subset M$, the map $P \mapsto \mu_P(E)$ is measurable,
\item  $\mu_{P}$ is $f-$invariant and ergodic for $\overline{\mu}$-almost every $P \in \mathcal{P},$
\item  $\mu(E) = \int \mu_P(E)\,d\overline{\mu}(P),$ for all measurable  $E  \subset M.$
\end{enumerate}
\end{theorem}

Here, $\overline{\mu}$ denotes the projected measure on $\mathcal{P}$ defined by:
\begin{itemize}
\item A subset $\mathcal{C} \subset \mathcal{P}$ is $\overline\mu-$measurable if $\displaystyle\bigcup_{P \in \mathcal{C} }P $ is  measurable in $M.$
\item $\overline{\mu }(\mathcal{C}) = \mu \left(\displaystyle\bigcup_{P \in \mathcal{C} }P\right).$
\end{itemize}

\section{Proof of Theorem \ref{mainA}}
We adapt the proof of  Theorem 5.2 of \cite{JHU} to the partially hyperbolic endomorphism  setting.. %{\color{red} inspirado por Plante}

Key observations:

\begin{enumerate}

\item Commuting rotations: For any fixed $\alpha \in \mathbb{S}^1,$ define the rotation map $$G_{\alpha}:  M \times \mathbb{S}^1 \rightarrow  M \times \mathbb{S}^1, \,\, G_{\alpha}(x, \theta) = (x, \theta + \alpha).$$
This satisfies the commutation relation:
$$G_{\alpha} \circ F_{\varphi} =  F_{\varphi} \circ G_{\alpha}.  $$

\item Invariance of stable foliations: The commutation above implies that the stable foliation $\mathcal{W}^s_{F_{\varphi}}$ is invariant under $G_{\alpha}$
\begin{equation} G_{\alpha}(\mathcal{W}^s_{F_{\varphi}}) = \mathcal{W}^s_{F_{\varphi}}\label{eq1}.
\end{equation}

\item Projection to the base: Let $p_1:  M \times \mathbb{S}^1 \rightarrow  M$ denote the projection onto the first coordinate.  Then
\begin{equation} p_1(\mathcal{W}^s_{F_{\varphi}}(x, \theta)) = \mathcal{W}^s_{f}(x)\label{eq2}.
\end{equation}
\end{enumerate}

Assume the stable foliation  $\mathcal{W}^s_{F_{\varphi}}$ is not minimal. Then there exists a stable leaf $W$ such that $\overline{W} \varsubsetneqq M \times\mathbb{S}^1 .$
Let $K \subset \overline{W}$ be a minimal set with the following properties:
\begin{enumerate}
\item $K \neq \emptyset,$
\item $K $ is closed,
\item $K$ is $s-$saturated (i.e., contains full stable leaves through any of its points).
\end{enumerate}

  By minimality, for any $(x,\theta) \in K$, the closure of its stable leaf satisfies $\overline{\mathcal{W}^s_{F_{\varphi}}(x,\theta)} = K.$

Note that $p_1(K) \supset p_1(\mathcal{W}^s_{F_{\varphi}}(x, \theta)) = \mathcal{W}^s_f(x), \,x \in M.$ Since $p_1(K)$ is compact and dense in $M,$ so $p_1(K) = M.$ We conclude that $K$ intersects every vertical fiber $\{y\} \times \mathbb{S}^1.$

\textbf{Claim 1:} The set $K$ intersects every vertical circle $\{y\} \times \mathbb{S}^1$ finitely many times.

Suppose there exists $y_0\in M$ such that  $K\cap (\{y_0\} \times \mathbb{S}^1) $ contains infinitely many points. Let $S_0 = \{y_0\} \times \mathbb{S}^1$.
So for any $\varepsilon > 0$, there exist $\theta_1, \theta _2\in \mathbb{S}^1$ with $ 0 < \theta_2 - \theta_1 < \varepsilon$ such that $(y_0, \theta_1),\, y_0, \theta_2)\in K \cap S_0$. Let $\alpha = \theta_2 - \theta_1.$

Consider the vertical rotation $G_{\alpha}(x, \theta)=(x, \theta+\alpha$. Then $$G_{\alpha} (y_0, \theta_1) = (y_0, \theta_2)\in K.$$
Since $K$ is $s-$saturated, closed, and minimal, $G_{\alpha}(K) \cap K \neq \emptyset$ implies $ G_{\alpha}(K ) = K .$ Similarly $ G_{-\alpha}(K ) = K .$

The orbit  $\{G_{n\alpha} (y_0 , \theta_1)\}_{ n \in \mathbb{Z}}$ is $\varepsilon -$dense in $S_0.$  Because $\varepsilon>0$ is arbitrary, $K\cap \mathbb{S}_0$ is dense in $\mathbb{S}_0$.

 By the property $(\ref{eq2})$ and minimality of $\mathcal{W}^s_f,$ we get that $K$ is $\varepsilon -$dense in $M\times \mathbb{S}^1,$ for arbitrary $\varepsilon > 0.$ Hence $K = M \times \mathbb{S}^1$.  This contradicts the assumption $\overline{W} \subsetneq M \times \mathbb{S}^1.$ Therefore, $K$ intersects every vertical circle finitely many times.

%{\color{red} parece usarse s\'olo que la foliaci\'on abajo es minimal}

\textbf{Claim 2:} The function $\psi: M \rightarrow \mathbb{N},$  defined by $\psi(y) = \#(K \cap \{y\}\times \mathbb{S}^1),$ is upper semi-continuous.

Assume $\psi$ is not upper semi-continuous.
Then there exist a point $y\in M$ and a sequence $\{y_n\}\subset M$ such that:
$$y_n\to y \,\,\,\,\text{ and }\,\,\,\,\, \limsup_{n\to\infty} \psi(y_n)=\eta>\psi(y).$$
Since $\psi$ takes valus in $\mathbb{N}$, $\eta\geq \psi(y)+1.$

  Since $K$ is compact, limit points of $K \cap \{y_n\}\times \mathbb{S}^1$ are points of $K \cap \{y\}\times \mathbb{S}^1.$
So, taking a subsequence if necessary,  there are $(y_n, \theta_{i,n}), i =1,2 \in K \cap \{y_n\}\times \mathbb{S}^1, $ such that $|\theta_{1,n} - \theta_{2,n}| \rightarrow 0,$ when $n \rightarrow +\infty.$

As in Claim 1, it leads us to a contradiction.

\textbf{Claim 3:} The function $\psi$ is constant.

Since $\psi$ is upper semi-continuous and $M$  is compact, $\psi$ achieve a maximum value $h.$
Let $y_0$ be a point such that $\psi(y_0) = h.$ We claim that the angle between consecutive points in $K \cap \{y_0\} \times \mathbb{S}^1$ is constant.  Indeed, call $\alpha_0$ the minimum angle between two consecutive points. Then, if there are two consecutive points with angle greater than $\alpha_0$, by composing with $G_{\alpha_0}$ we will get a new point and $\psi(y_0) > h.$ Then the angle between two consecutive points in $K \cap \{y_0\} \times \mathbb{S}^1$  is constantly equal to  $\alpha_0$. %{\color{red} el argumento es m\'as simple, si el \'angulo no es constante, generamos m\'as puntos.}

By the property $(\ref{eq2})$ and $G_{\alpha_0}(K) = K,$ we conclude that $\psi(y) \geq h,$ for all $y\in M$.  Since $h$ is the maximum,  we get $\psi(y) = h$ for all $y\in M$. This proves Claim 3.
\medskip

As a consequence we get that  $(K, p_1) $ is an $h-$fold covering of $M$. Then
$$\mathcal{K} = \displaystyle\bigcup_{\alpha \in \mathbb{S}^1} G_{\alpha}(K)$$
is  a $C^0-$foliation  of $M\times \mathbb{S}^1.$

\textbf{Claim 4:} For every $(x, \theta) \in M \times \mathbb{S}^1,$ all unstable leaves through $(x,\theta)$ intersect a unique leaf of $\mathcal{K}.$

Assume Claim 4 is false. Lift the system to  the universal cover of $M \times \mathbb{S}^1,$ and let $\overline{\mathcal{K}}$ and  $F$ denote the lifts of $\mathcal{K}$ and  $F_{\varphi}$, respectively.  In $\widetilde{M} \times \mathbb{R}$,  there exist distinct leaves $K_1, K_2\subset \overline{\mathcal{K}}$ connected by an unstable segment $[a,b]^u$ with $a \in K_1$ and $b \in K_2.$ Since $F$ acts isometrically along  the  $c-$coordinate, we have:

$$dist_c(F^n (K_1), F^n(K_2)) \nrightarrow 0,\,\,\,\,\,\,\, \text{ as } n \rightarrow \pm \infty,$$
where $dist_c(K_1, K_2) = min\{dist(k_1, k_2)| p_1(k_1) = p_1(k_2)\}.$

\begin{figure}[H]
\centerline{
\includegraphics[width=12cm]{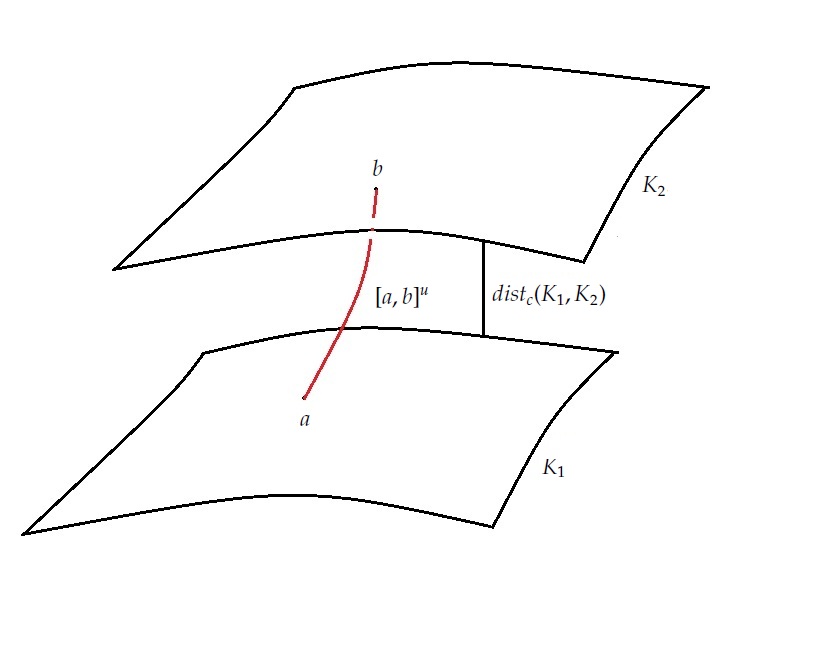}}
\caption{{\small The leaves $K_1$ and $K_2$ connected by an unstable leaf}.}
%\label{figsemi}
\end{figure}

However, the unstable segment $[a,b]^u$ contracts under backward iteration:
$$dist(F^{-n}(a),F^{-n}(b)) \rightarrow 0, \,\,\,\,\,\,\, \text{ as }  n \rightarrow +\infty.$$
This implies
\begin{equation}\label{nzero}
dist_c(F^n (K_1), F^n(K_2)) \rightarrow 0, \,\,\,\,\,\,\, \text{ as } n \rightarrow  +\infty,
\end{equation}
contradicting the earlier assertion. Thus, Claim 4 holds.

\section{Some Consequences of Theorem \ref{mainA}}

For subsequent results, we utilize accessibility--a central concept in partially hyperbolic dynamics. In diffeomorphisms, accessibility is an important tool to generalize the Hopf argument to prove ergodicity. For endomorphisms, this concept was rigorously developed in  \cite{He}.

\begin{definition}\label{accdef}
A partially hyperbolic endomorphism   $f: M \rightarrow M$  is  accessible if for any pair of points $x,y \in M$,  there exists a  piecewise $C^1-$path $\gamma_1 \ast \gamma_2 \ast \ldots \ast \gamma_n$ connecting $x,y,$ such that:
\begin{enumerate}
 \item Each path $\gamma_j:[0,1] \rightarrow M$, $(j = 1, \ldots, n)$ satisfies:
 \begin{itemize}
 \item $\gamma_j([0,1]) \subset  W^s_f(z_{j-1})$, or
 \item $\gamma_j([0,1]) \subset  W^u_f(\tilde{z}),$ where $\tilde z$ is a lift of $z_{j-1}$ (i.e., $p(\tilde{z})= z_{j-1}).$
 \end{itemize}
 \item The path connects  $x = z_0, \ldots , z_n = y$ with $\gamma_j(0) = z_{j-1}$ and $\gamma_j(1) = z_j$.
 \end{enumerate}

\end{definition}

\begin{figure}[H]
\centerline{
\includegraphics[width=12cm]{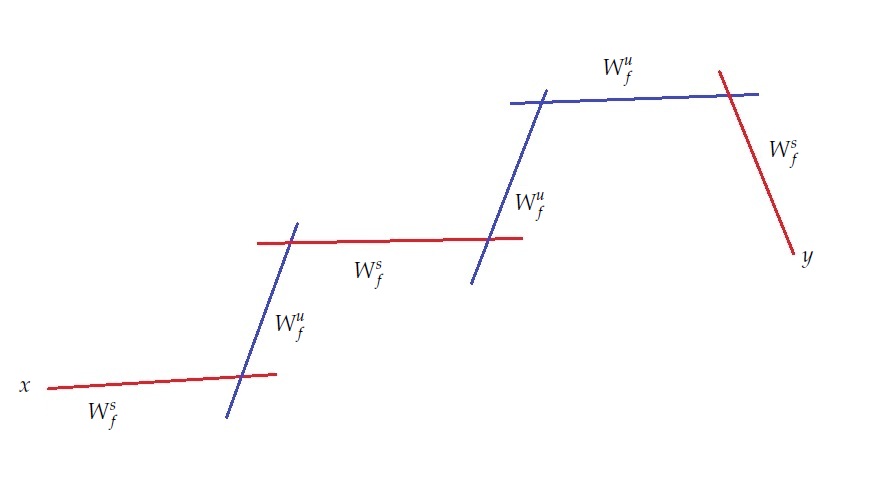}}
\caption{{\small A $su-$path connecting $x$ and $y$}.}
%\label{figsemi}
\end{figure}

Under the assumptions of Theorem $A,$ the non-integrability of $E^s \oplus E^u$ ensures accessibility, as a consequence of the geometric structure  of unstable and stable manifolds.

\begin{theorem} Let $F_{\varphi}$ be as in Theorem A. If $E^s \oplus E^u$ is not integrable, then $F_{\varphi}$ is accessible.
\end{theorem}

\begin{proof}

We analyze two geometric configurations arising from the non-integrability of $E^s \oplus E^u,$ adapting Brin's quadrilateral argument  \cite{Br}.

\noindent\textbf{Configuration 1 (Non-integrable in the universal cover):} \newline If $E^s \oplus E^u$ is not integrable in the universal cover, Brin's quadrilateral configuration applies (Fig. \ref{quadrilatero}). By the minimality of $\mathcal{W}^s_{F_{\varphi}}$, stable and and unstable paths connect any two points, ensuring accessibility.

\begin{figure}[H]
\centerline{
\includegraphics[width=13cm]{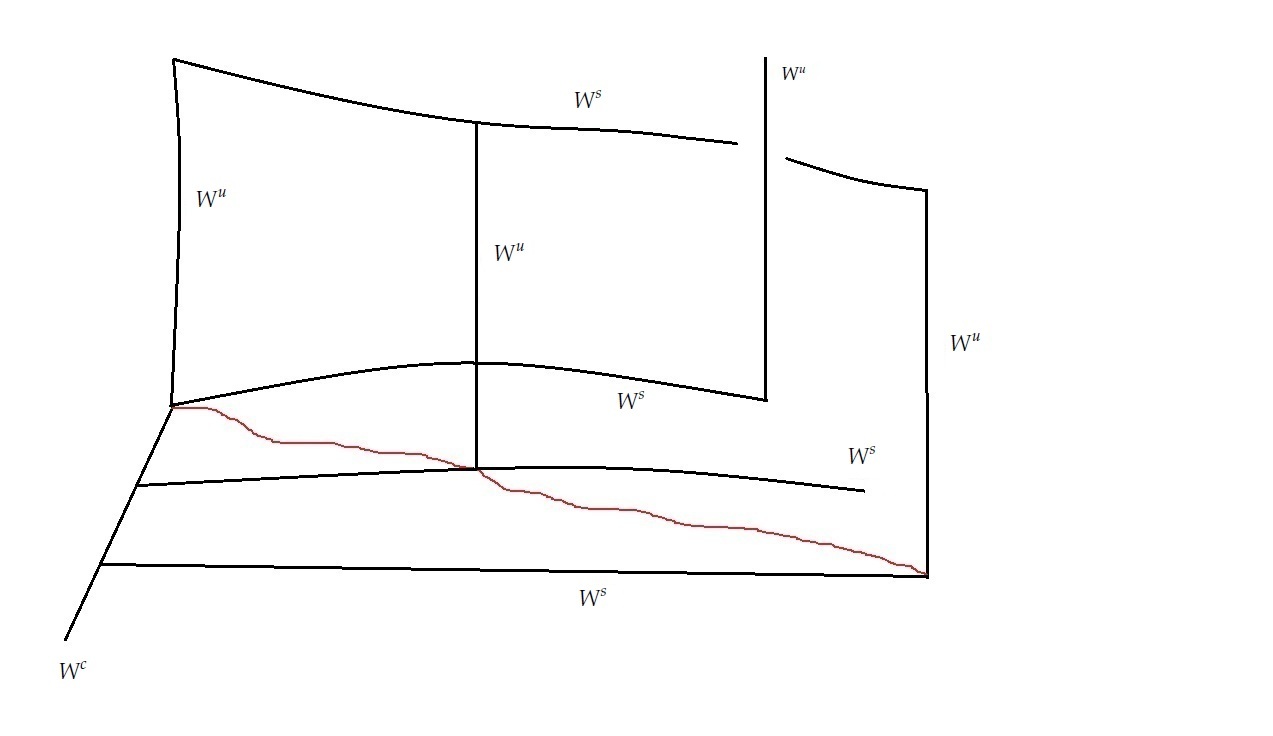}}
\caption{{\small Situation 1: Brin's quadrilateral configuration}.}
\label{quadrilatero}
\end{figure}

\noindent\textbf{Configuration 2 (Non-integrable in $\mathbb{T}^2 \times \mathbb{S}^1$) :} \newline If $E^s \oplus E^u$ is integrable in the universal cover but  not  in $\mathbb{T}^2 \times \mathbb{S}^1,$  consider two leaves $\mathcal{W}^s_{F_{\varphi}}(x)$ and $\mathcal{W}^s_{F_{\varphi}}(y).$ By minimality of $\mathcal W^s_{F_\varphi}$,  these  leaves approach the ``walls'' of the $su-$leaves of the figure $(\ref{twoside})$. Selecting $su-$leaves from $\mathcal{W}^s_{F_{\varphi}}(x)$ and $\mathcal{W}^s_{F_{\varphi}}(y),$ near these walls,  continuity ensures intersection, establishing accessibility.
\begin{figure}[H]
\centerline{
\includegraphics[width=13cm]{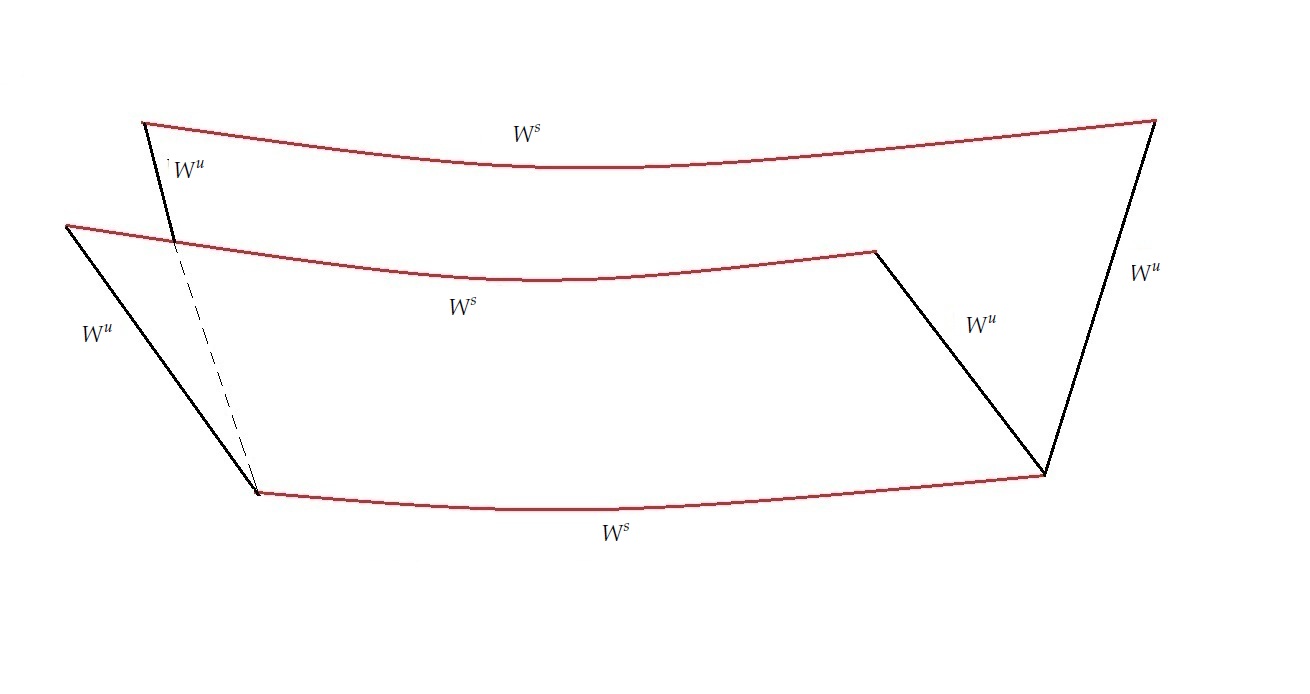}}
\caption{{\small Situation 2: Two $su-$leaves from the same stable leaf}.}
\label{twoside}
\end{figure}

\end{proof}

\section{Proof of Theorem \ref{accBrin}}

The proof here is an adaptation of the argument in \cite{Br}.

\begin{proof}
The proof adapts Brin's argument \cite{Br} to partially hyperbolic endomorphisms, leveraging accessibility and volume preservation.

Since $f$ is volume-preserving, The Poincar\'e Recurrence Theorem implies $\Omega(f) = M$ and $m(Rec(f)) = 1,$ wher  $Rec(f)$ denotes the set of recurrent points $x \in M$ (i.e.,  $x \in \omega (f)$).

To prove topological transitivity, we show that for any  $x, y \in M$ and  $\varepsilon > 0,$ there exists $k = k(x,y, \varepsilon)$ such that $f^k(B_{\varepsilon} (x)) \cap B_{\varepsilon}(y) \neq \emptyset. $ Since $f$ is accessible, there is a $C^0-$path connecting $x$ and $y,$ composed of  stable/unstable segments.

Theorem \ref{accBrin} is a consequence of the two following  lemmas.

\begin{lemma}\label{accs}
Let $x = z_0, \ldots, z_n = y$ be as in Definition \ref{accdef} and let $z_{i+1} \in W^s(z_i)$. If  every neighborhood $z_i\in U$ intersects $f^k(B_{\varepsilon} (x)) $ for some $k\geq 0$, then   $f^{\hat{k}}(B_{\varepsilon} (x)) \cap B_{\varepsilon} (z_{i+1}) \neq \emptyset$ for some $\hat{k} \geq 0.$
\end{lemma}

The proof mirrors Lemma 2 of \cite{Br}, relying  on forward iterations  and the stable foliation's uniformity (well-defined for all $x\in M$, independent of  $\tilde{x} \in M^f $).

The key challenge lies in handling unstable leaves, as partially hyperbolic endomorphisms do not necessarily admit a well-defined unstable foliation on $M$. Unlike stable leaves, which are uniquely determined at each
$x\in M$, unstable leaves depend on the choice of preorbit
$\tilde x\in M^f$ (with
$p(\tilde x)=x$). To address this, we project unstable leaves from the universal cover
$\widetilde{M}$, where the lifted dynamics
$\tilde f$ behaves as a diffeomorphism, ensuring locally consistent unstable manifolds.

\begin{lemma}\label{accu}
Let $x = z_0, \ldots, z_n = y$ be as in Definition \ref{accdef}. Suppose that $z_{i+1} \in W^u(\tilde{z}_i),$ such that $p(\tilde{z}_i) = z_i,$ and for any open set $U \in M,$ with $z_i \subset U,$ there is $k = k(U) \geq 0,$ integer, such that $f^k(B_{\varepsilon} (x)) \cap U \neq \emptyset.$    Then there is $\hat{k} \geq 0,$ integer, such that $f^{\hat{k}}(B_{\varepsilon} (x)) \cap B_{\varepsilon} (z_{i+1}) \neq \emptyset.$
\end{lemma}

\begin{proof}
First, we recall that for the partially hyperbolic system $f$  the unstable manifolds $W^u_f(\tilde{x}),$  where
$
\tilde{x} = (\ldots, x_{-2}, x_{-1}, x_0, x_1, x_2, \ldots) \in M^f,
$
can be constructed using the graph method, as outlined in \cite{HPS}. Moreover, the lifts of $\tilde{f}$ are partially hyperbolic diffeomorphisms of $ \widetilde{M}, $ \cite{micsa}, for which both unstable and stable manifolds can be constructed. Additionally, the complete orbits of $ \tilde{f}$ project onto complete orbits of $M^f.$

Fixing a lift $\tilde{f}$, we note that unstable manifolds vary $C^1-$continuously by the graph method. The density in $M^f$ of orbits projected from orbits of $\widetilde{M}^{\tilde{f}}$ ensures that the unstable manifolds of  $f$ can be $C^1-$approximated by the projections of the unstable manifolds of $\tilde{f}$ onto $M$ from the universal cover $\widetilde{M}.$ Given a fixed $R > 0,$ this approach can be carried out by considering unstable discs of constant radius.

Project an unstable leaf $W^u(z_i) = \pi(W)$ from the universal cover $\tilde M$, ensuring it is  $C^1-$close to $W^u(\tilde{z}_i)$ and intersects $B_{\varepsilon}(z_{i+1})$ at a point $w_{i+1},$ very close to $z_{i+1}.$
%{\color{red} ac\'a estamos usando la densidad de las inestables que vienen del cubrimiento universal. Hay que decirlo y citar literatura. }

Let $U$ be a small enough neighborhood of $z_i$ contained in $f^k(B_{\varepsilon}(x)).$ Project  of a  $u-$foliated  set from $\tilde M$ to $M$, such that unstable leaves starting in $U$ cross $B_{\varepsilon }(z_{i+1}).$ By continuity, there exists  $R > 0,$ such that every $R-$radius unstable disc centered in $U$ intersects $B_{\varepsilon }(z_{i+1}).$

For $\hat{z}_i = (\ldots z_{i, -3}, z_{i, -2}, z_{i, -1}, z_{i, 0}, z_{i, 1}  z_{i, 2}, \ldots)\in M^f$ such that $p(W^u(\hat{z}_i)) = W^u(z_i),$ choose $N > 0$  and $V_N\subset M$ an open neighborhood of $z_{i, -N}$ such that $f^N(V_N) = U.$ Let $V_N' = Rec(f) \cap V_N.$ Since $f$ is a local diffeomorphism, we get $m(U \setminus f^N(V_N')) = 0.$ Moreover, taking $U' = Rec(f) \cap U,$ we get $m(U \setminus U') = 0.$ Consequently $m(U'\setminus f^N(V')) = 0.$ Finally, take the set $ U'' =  Rec(f) \cap U' \cap f^N(V') .$ By Poincar\'e Recurrence Theorem $m(U \setminus U'') = 0.$ Hence, for $m-$almost every point $ \theta\in U$  there is an increasing sequence $(n_k)_{k=1}^{+\infty}$ of positive integers, such that $f^{n_k}(\theta) \in U'' $ and consequently $f^{n_k - N} (\theta) \in V_N.$

To finish the proof, choose a point $\theta \in U'' $ and $W^u(\theta) \cap U$ a connected component of the leaf of the foliation previously constructed containing $\theta.$ Consider the unstable submanifolds
$ f^{n_k - N} (W^u(\theta) \cap U).$ Since $N$ is arbitrarily  large, we obtain that the $R-$disc of $f^N (f^{n_k - N} (W^u(\theta) \cap U))$ centered in $f^{n_k}(\theta)$ is $C^1-$close to the $R-$disc of $W^u(f^{n_k}(\theta)) $ centered in $f^{n_k}(\theta),$ which  is $C^1-$close to the $R-$disc $W^u(z_i),$ centered in $z_i.$ So, for large values of $j > 0,$ we get $f^{n_j}(W^u(\theta) \cap U) \cap B_{\varepsilon}(z_{i+1}) \neq \emptyset.$

Finally,  we get $f^{k + n_j}(B_{\varepsilon}(x) ) \cap B_{\varepsilon}(z_{i+1})  \neq \emptyset ,$  for large $j > 0.$
\end{proof}

Theorem \ref{accBrin} follows by iteratively applying  Lemmas \ref{accs} and \ref{accu} along the $C^0-$path from $x$ to $y$.
\end{proof}

\section{Proof of Theorem \ref{mainB}} \begin{proof}

Our strategy to establish ergodicity is to prove that the ergodic decomposition of $m$ consists of  a single measure.

Let $\{\mu_{P}\}_{\alpha \in \mathcal{P}}$ be the ergodic decomposition of the measure $m.$ Using Ruelle's inequality is not hard to check:

\begin{lemma} \label{lem1} For $\overline{\mu}$ a.e. $P \in \mathcal{P},$ the measures $\mu_{P}$ is $SRB.$
\end{lemma}

\begin{proof} Let $\lambda^u_{F_{\varphi}}(x,\theta)=\displaystyle\sum_{\lambda_{F_{\varphi}}(x,\theta)>0} \lambda_{F_{\varphi}}(x,\theta).$ By Pesin's formula $h_{m}(F_{\varphi}) = \int_{\mathbb{T}^3}  \lambda^u_{F_{\varphi}}(x,\theta) dm (x,\theta)$.
By the Jacobs formula
$$ h_m(F_{\varphi}) = \int_{\mathcal{P}} h_{\mu_{P}} (F_{\varphi}) d\overline{\mu}(P) \underbrace{\leq}_{Ruelle's\, inequality}  \int_{\mathcal{P}} \lambda^u_{F_{\varphi}}(\mu_{P}) d\overline{\mu}(P) = \int_{\mathbb{T}^3} \lambda^u_{F_{\varphi}}(x,\theta) dm(x,\theta).  $$
We  conclude that for $\overline{\mu}$ almost every $P,$ we have $h_{\mu_{P}}(F_{\varphi}) = \lambda^u_{F_{\varphi}}(\mu_{P}),$ meaning that $\mu_{\alpha} $ is $SRB$ for $F_{\varphi}.$
\end{proof}

Consider $\mu_P$ an ergodic SRB $F_{\varphi}$-invariant component of $m$. Since $F_{\varphi} \circ G_{\theta} = G_{\theta} \circ F_{\varphi}$, we have
$$ (G_{\theta})_{\ast} (\mu_P) = (G_{\theta})_{\ast}((F_{\varphi})_{\ast}(\mu_P)) = (F_{\varphi})_{\ast}((G_{\theta})_{\ast}(\mu_P)), $$
so $\mu_{\theta P}:= (G_{\theta})_{\ast} (\mu_P)$ is a conditional measure of $m$, as $(G_{\theta})_{\ast}(m) = m$. Moreover, since  the center foliation $\mathcal{F}^c$ is absolutely continuous, it follows that $\mu_{\theta P}$ is SRB if $\mu_P$ is SRB.

As $\mu_P$ is SRB, there exists an open disk $D^u$ such that $D^u \subset B_{\mu_P}$ (the basin of $\mu_P$). Additionally, the $s$-saturation $Sat^s(D^u)$ is contained in $B_{\mu_P}$. We note that if $x \in D^u \cap B_{\mu_P}$, then $x' = G_{\theta}(x) \in B_{\mu_{\theta P}}$. Indeed, given $g: M \times \mathbb{S}^1 \rightarrow \mathbb R$ a continuous map, and considering that $F_{\varphi}^j \circ G_{\theta} = G_{\theta} \circ F_{\varphi}^j$ for any $j \geq 0$, we obtain:

$$\frac{1}{n} \sum_{j=0}^{n-1} g(F_{\varphi}^j(G_{\theta}(x))) =  \frac{1}{n} \sum_{j=0}^{n-1} g \circ G_{\theta}(F_{\varphi}^j(x)),$$
so, as $n \rightarrow + \infty$, we get

$$\displaystyle\lim_{n \rightarrow + \infty} \frac{1}{n} \sum_{j=0}^{n-1} g(F_{\varphi}^j(G_{\theta}(x))) = \displaystyle\lim_{n \rightarrow + \infty}
\frac{1}{n} \sum_{j=0}^{n-1} g \circ G_{\theta}(F_{\varphi}^j(x)) = \displaystyle\int_{\mathbb{T}^2 \times \mathbb{S}^1 } g \circ G_{\theta} d\mu_P  =
 \displaystyle\int_{\mathbb{T}^2 \times \mathbb{S}^1 } g d\mu_{\theta P}.$$

We claim that $\mu_P$ is the unique ergodic component of $m$. Suppose $\mu_Q$ is another ergodic SRB component of $F_{\varphi}$. As before, consider $D'^u$ an unstable disk contained in $B_{\mu_Q}$. The $s$-saturation $Sat^s(D'^u)$ is a dense subset of $M\times \mathbb{S}^1$. Considering an open rectangle $R \subset Sat^s(D^u)$ such that $D^u \subset R$, the density of $Sat^s(D'^u)$ and the absolute continuity of $\mathcal{F}^c$ allow us to conclude that, for a small angle $\theta$, the map $G_{\theta}$ sends a point $x \in Sat^s(D'^u)$ to a point $x' = G_{\theta}(x) \in R$. By construction, $x' \in B_{\mu_P} \cap B_{\mu_{\theta Q}}$.

Since $\mu_P$ and $\mu_Q$ are distinct ergodic components of $m$, there exists a continuous function $\psi: M \times \mathbb{S}^1 \rightarrow M$ such that

$$\delta = \left| \displaystyle\int_{M \times \mathbb{S}^1} \psi d\mu_{P} - \displaystyle\int_{M \times \mathbb{S}^1} \psi d\mu_{Q} \right| > 0. $$

Consider $I_P =  \displaystyle\int_{M \times \mathbb{S}^1} \psi d\mu_{P}$, $I_Q =  \displaystyle\int_{M \times \mathbb{S}^1} \psi d\mu_{Q}$, and finally $I_{\theta Q} =  \displaystyle\int_{M \times \mathbb{S}^1} \psi d\mu_{\theta Q} =  \displaystyle\int_{M \times \mathbb{S}^1} \psi \circ G_{\theta} d\mu_{Q}.$

As before, the angle $\theta$, which defines $x' = G_{\theta}(x) \in R$, can be made arbitrarily small, so that $|I_Q - I_{\theta Q}| < \frac{\delta}{2}$, since $G_{\theta}$ is $C^1$-close to the identity.

On the other hand, since $x' \in B_{\mu_P} \cap B_{\mu_{\theta Q}}$, we have $\mu_P = \mu_{\theta Q}$, and thus $I_P = I_{\theta Q}$.

Finally,

$$\delta = |I_P - I_Q| = |I_{\theta Q} - I_Q| < \frac{\delta}{2},$$
a contradiction.

Thus, the decomposition of $m$ into ergodic components is unique, and therefore $m$ is ergodic.

\end{proof}

\section{Proof of Theorem \ref{mainD}}

The strategy is similar to the proof of Theorem \ref{mainB}.

\begin{proof}
In the present Theorem, accessibility means $u-$accessibility; that is, each component in the path in Definition \ref{accdef}   lies in an unstable manifold.  Let $m$ be the volume normalized measure preserved by $F$, and $\{\mu_P\}_{P \in \mathcal{P}}$ be the ergodic decomposition of $m.$ As in Lemma \ref{lem1},  the ergodic component $\mu_P$ is $SRB,$ for $\bar{\mu}$ a.e. $P.$ We will prove that the decomposition of $m$ is a singleton. Let $\mu$ and $\nu$ be two distinct ergodic conditional $SRB-$measures of $m.$ There is a continuous function $\psi: M \times \mathbb{S}^1\rightarrow \mathbb{R}$, such that $\displaystyle\int_{M \times \mathbb{S}^1} \psi d\mu \neq \displaystyle\int_{M \times \mathbb{S}^1} \psi d\nu.$

Take $\delta = \left|\displaystyle\int_{M \times \mathbb{S}^1} \psi d\mu - \displaystyle\int_{M \times \mathbb{S}^1} \psi d\nu\right|  > 0.$ Since $\mu$ and $\nu$ are $SRB$ components, there exist open $u-$disks $W_1$ and $W_2,$ such that $W_1 \subset B_{\mu}$ and $W_2 \subset B_{\nu}.$ Let $\varepsilon > 0$ be small enough and consider the following tubes

$$T_i = \bigcup_{-\varepsilon < r < \varepsilon}\{W_i + r\}, i = 1,2,$$
where the additive notation means a translation in the vertical direction (tangent to $\mathbb{S}^1$). The disks $T_1$ and $T_2$ are open sets and foliated by unstable manifolds $W_i + r, i = 1,2$ respectively.

As before, if we define $\mu_r = (G_r)\ast(\mu),$ with  $-\varepsilon < r < \varepsilon,$ we obtain $$\left|\displaystyle\int_{M \times \mathbb{S}^1} \psi d\mu_r - \displaystyle\int_{M \times \mathbb{S}^1} \psi d\mu \right| < \frac{\delta}{4},$$ for $\varepsilon > 0$ small enough.

The endomorphism $F$ is topologically transitive, then there is $n > 0$ such that $F^n(T_1) \cap T_2 \neq \emptyset.$ Then there is a small $r \in (-\varepsilon , \varepsilon),$ for which is well defined a center holonomy between a small disk $W_1^n$ of $F^n(W_1 + r)$ and $W_2.$  As before, there is a point $x \in B_{\mu_r}$ such that  $x = G_s(x'), x' \in W_2,$ and  $x \in B_{\mu_r} \cap B_{\nu_s},$ where $r,s \in (-\varepsilon , \varepsilon). $

So $\mu_r = \nu_s$ and $\left|\displaystyle\int_{M \times \mathbb{S}^1} \psi d\nu_s - \displaystyle\int_{M \times \mathbb{S}^1} \psi d\mu \right| < \frac{\delta}{4}.$ On the other hand

 $\left|\displaystyle\int_{M \times \mathbb{S}^1} \psi d\nu_s - \displaystyle\int_{M \times \mathbb{S}^1} \psi d\mu \right| \geq \left|\displaystyle\int_{M \times \mathbb{S}^1} \psi d\nu - \displaystyle\int_{M \times \mathbb{S}^1} \psi d\mu \right| - \left|\displaystyle\int_{M \times \mathbb{S}^1} \psi d\nu_s - \displaystyle\int_{M \times \mathbb{S}^1} \psi d\nu \right| > \delta - \frac{\delta}{4} = \frac{3\delta}{4},$ a contradiction.

 We conclude that the system $(F, m)$ is ergodic.

\end{proof}

\section{Examples}

Let $A: \mathbb{T}^2 \rightarrow \mathbb{T}^2,$ such that $A(x,y) = (3x +  y, x + y) (mod \, \mathbb{Z}^2).$  Let $U$ be a small neighborhood of $0.$ The map $A$ is a volume-preserving Anosov endomorphism with degree equal to two. Since we will deal with local perturbations we can consider that $E^s_A = \langle\ e_1 \rangle$ and $E^u_A = \langle e_2 \rangle,$ where $e_1, e_2$ are the canonical vectors.  Construct a $C^{\infty}-$function $\varphi: \mathbb{T}^2 \rightarrow \mathbb{S}^1$  satisfying:

\begin{enumerate}
  \item $\varphi$ is $C^1-$close to the constant zero function.
   \item $\varphi(0) = 0.$
  \item $\varphi$ is null in $\mathbb{T}^2\setminus U.$
  \item $\frac{\partial \varphi}{\partial{y}}(0) \neq 0.$
\end{enumerate}

The map $F_{\varphi} ((x,y), \theta)  = (A(x,y), \theta + \varphi(x,y))$ where $(x,y) \in \mathbb{T}^2, \theta\in \mathbb{S}^1$ is a volume-preserving partially hyperbolic endomorphism with degree two. We observe that $O = (0,0,0)$ is a fixed point of $F_{\varphi}.$ Moreover,  restricted to $(0,0)\times \mathbb S^1,$ $F_{\varphi}$ acts as the identity.  Then, at the point $O,$ we have
$E^s =  \langle\ e_1 \rangle, E^u =  \langle\ e_2 \rangle$ and $E^c =  \langle\ e_3 \rangle.$

 For every $z \in \mathbb{T}^3$, at least one preimage by $F_{\varphi}$ lies out of $U,$  since $F_{\varphi}$ is locally injective. So, the exists a sequence $ \tilde{z} = (\ldots, z_{-3}, z_{-2}, z_{-1}, O,O,O \ldots ) \in M^{F_{\varphi}},$  such that $z_{-n} = (p_{-n}, 0), n \geq 1,$ such that $ p_{-n} \notin U,$ for any $n \geq 2,$ also $p_{-1} = 0.$

 The action of $F_{\varphi}$ is the same as $A \times Id_{\mathbb{S}^1},$ for each $z_{-n}, n \geq 2.$  Then, on the point $z_{-1},$ we obtain
$E^s =  \langle\ e_1 \rangle, E^u =  \langle\ e_2 \rangle$ and $E^c =  \langle\ e_3 \rangle.$

We have $$DF_{\varphi}(x,y,\theta) = \left(
                                       \begin{array}{ccc}
                                         a_s & 0 & 0 \\
                                         0 & a_u & 0 \\
                                         \frac{\partial \varphi}{\partial x}(x,y) & \frac{\partial \varphi}{\partial y}(x,y) & 1 \\
                                       \end{array}
                                     \right),
$$ where $a_s$ and $a_u$ are the constant of contraction and expansion of $A,$ respectively.  By the partially hyperbolic structure at  $\tilde{z},$ we have $E^u = \langle a_u \cdot e_2 + \frac{\partial \varphi (p)}{ \partial y } \cdot e_3 \rangle $  at $O = (0,0,0)$.
So $E^u$ has a vertical component different from $E^u$ corresponding to the orbit constant equal to $O $ for each coordinate $n \in \mathbb{Z}.$ Hence, $E^s \oplus E^u$ is not integrable.
By Theorem A, the stable foliation of  $F_{\varphi}$ is minimal, and consequently $F_{\varphi}$ is accessible.  Since $F_{\varphi}$ is $m-$preserving, Theorem \ref{mainB} implies that the system $(F_{\varphi}, m)$ is ergodic.

\section {Proof of Theorem \ref{mainC}}
\begin{proof}[Proof of Theorem \ref{mainC}]

First, we observe that $\displaystyle\int_M \lambda^c_f(x) dm(x) > 0$ implies $Vol(\Lambda^c_+(f)) > 0,$ where $\Lambda^c_+(f)  = \{ x \in M|\; \lambda^c_f(x) > 0 \}$. Consider $g$ the restriction $g= f|_{\Lambda^c_+(f)}: \Lambda^c_+(f) \rightarrow \Lambda^c_+(f)$ and $\mu $ such that $\mu(B) = \frac{m(B \cap \Lambda^c_+(f))}{m(\Lambda^c_+(f))}.$ Of course $g$ preserves $\mu.$

Moreover, by the Pesin's Formula, $m$ is a SRB measure for $f.$

As before, every ergodic component of $m$ is a $SRB$ measure. We define the set $\mathcal{E}^+ =\{ \mu \in \mathcal{M}_f |\; \mu \mbox{is a $SRB$ ergodic component of $m$ such that } \mu(\Lambda^c_+(f)) = 1 \}.$

Given $\mu \in \mathcal{E}^+ $ consider $\hat{\mu}$ the unique borelian measure in $M^f$ such that $p_{\ast}(\hat{mu}) = \mu.$ For $\hat{\mu}$ almost every $\tilde{x} \in M^f,$ $\dim (W^{+}(\tilde{x})) = d_u + 1.$

For a given subordinated partition $\xi$ with respect to $W^+$ of $\hat{\mu},$ holds

\begin{equation}\label{ae}
\hat{\mu}_{\tilde{x}}^{\xi} (p^{-1} (\Lambda^c_+(f) \cap B_{\mu}) ) = 1, \hat{\mu}- a.e. \; \tilde{x} \in M^f,
\end{equation}
where $B_{\mu}$ denotes the basin of $\mu.$ Consequently there in a disk $D^+ \subset W^+(\tilde{x})$ such that $Leb_{D^+}$ a.e $y \in D^+$ lies in $\Lambda^c_+(f) \cap B_{\mu}, $ for a taken $\tilde{x}$ as in $(\ref{ae}).$

Let $D^+_s$ the $s-$saturation of $D^+ \cap B_{\mu}\cap \Lambda^c_+(f) .$ Since $B_{\mu}$  and $\Lambda^c_+(f)$ are $s-$saturated sets, we get $D^+_s \subset B_{\mu}$ and  $D^+_s \subset \Lambda^c_+(f),$ in particular the sets $\Lambda^c_+(f)$ and $B_{\mu}$ are open and dense sets, up to a set $Z,$ such that $m(Z) = 0.$ As a conclusion of this, the set $\mathcal{E}^+$ is a unit set. Thus  $ \mathcal{E}^+ = \{\nu\},$ such that $\nu(B) = \frac{m(B \cap \Lambda^c_+(f) )}{m(\Lambda^c_+(f))},$ and the restriction $f|_{\Lambda^c_+(f)}$ is ergodic, with respect to $\nu.$

Let $U \subset M,$ an open and dense set such that $m( (\Lambda^c_+(f) \setminus U) \cup( U \setminus \Lambda^c_+(f) ) = 0.$ Consider the compact set $K = M\setminus U.$ We will show that $m(K) = 0.$

Suppose that $m(K) > 0,$ in particular, by definition $\lambda^c_f(x) \leq 0,$ for $m-$a.e. $x \in K,$ moreover, $K$ is $s-$saturated, up to a null volume subset. Let  $K^s$ be the $s-$saturation of $K,$ we have $m(K^s) > 0.$ Since $\mathcal{F}^s$ is minimal and compactness of $K,$ it is possible to find a uniform size $R> 0,$ such that the $s-$disks centered in points of $K$ enters in $U.$ In this way, taking into account $\mathcal{F}^s$ is an absolutely continuous foliation and $m(K) > 0,$ then  $K^s$ enters in $U$ with positive measure, and we could conclude $m(U \cap K) = m(U \cap K^s) > 0,$ a contradiction. Thus $f$ is ergodic.

It remains to prove in the case $(2)$ of the statement of Theorem \ref{mainC}, that $f$ is robustly ergodic, meaning that, there is a $C^1-$neighborhood $\mathcal{U}$ among $f$ of all $m-$preserving partially hyperbolic endomorphism, such that every $g \in \mathcal{U}$ a $C^2-$endomorphism hods that $g$ is ergodic with respect to $m.$

From \cite{micsa}, given an arbitrary partially hyperbolic endomorphism $g: M \rightarrow M,$ the bundles $E^{cs}_g(x), E^s_g(x)$ are uniquely defined for each $x \in M,$ independent of the $\tilde{x} \in M^g$ such that $p(\tilde{x}) = x.$ Combining the Birkhoff Ergodic and Oseledec's Theorem

\begin{equation}\label{center_exponent}
\begin{aligned}
\int_M \lambda^c_g(x) \, dm(x) &= \int_M (\lambda^c_g(x) + \lambda^s_g(x)) \, dm(x) - \int_M \lambda^s_g(x) \, dm(x) \\
&= \int_M J^{cs}g(x) \, dm(x) - \int_{M} J^sg(x) \, dm(x).
\end{aligned}
\end{equation}

From \cite{micsa}, the bundles $E^{cs}_g$ and $E^s_g$ varies continuously with $x \in M$ and with respect to the endomorphism.  Now, using $(\ref{center_exponent}),$ in a small neighborhood $\mathcal{U}$ of $f$ among all $m-$preserving partially hyperbolic endomorphism, we get $\int_M \lambda^c_g(x)dm(x) > 0,$ for any $g \in \mathcal{U}.$ If $\mathcal{U}$ is enough small around $f,$ a $C^2$ and $m-$preserving $g \in \mathcal{U}$ has minamal stable foliation, hence $g$ is ergodic, as proved in the first part.

\end{proof}

It is interesting to note that Theorem $\ref{mainC}$ holds also for diffeomorphisms. We can state the following corollary.

\begin{cor}
 Let $f: M \rightarrow M$ be a $C^2-$volume preserving partially hyperbolic diffeomorphism, with $dim(E^c)=1$ \, and\,  $\dim(E^{\ast} ) \geq 1, \ast \in \{s,u\}. $ Suppose that $\displaystyle\int_{M} \lambda^c_f(x) dm(x) > 0.$
\begin{enumerate}
\item If $\mathcal{F}^s_f$ is minimal, then $f$ is ergodic.
\item If $\mathcal{F}^s_f$ is robustly minimal, then $f$ is stably ergodic.
\end{enumerate}
\end{cor}

%%%%%%%%%%%%%%%%%%%%%%%%%%%%%%%%%%%%%%%%%%%%%%%%%%%%%%%%%
%%%%%%%%%%%%%%%%%%%%%%%%%%%%%%%%%%%%%%%%%%%%%%%%%%%%%%%%%

%--------------------------------------------------------------------------------------

\end{document}